\documentclass[11pt,reqno,a4paper]{amsart}
\usepackage{amssymb}
\usepackage{amsfonts}
\usepackage{amsmath, amsthm, amssymb}
\usepackage[english]{babel}
\usepackage[mathcal]{eucal}

\usepackage{pdfsync}
\allowdisplaybreaks

\title[Uniformly continuous linear extension in TVS]{A uniformly continuous linear extension principle in topological vector spaces with an application to Lebesgue integration}
\author{Ben Berckmoes}
\address{Departement Wiskunde-Informatica, Middelheimcampus, Middelheimlaan 1, 2020 Antwerp, Belgium}
\email{ben.berckmoes@ua.ac.be}
\date{}

\thanks{The author is PhD fellow at the Fund for Scientific Research of Flanders (FWO)}

\DeclareMathOperator*{\mymax}{ma\vphantom{p}x}

\begin{document}

\maketitle

\hyphenation{ex-ten-da-ble}
\hyphenation{Pro-po-si-tion}
\hyphenation{Ge-bie-te}

\newtheorem{pro}{Proposition}[section]
\newtheorem{lem}[pro]{Lemma}
\newtheorem{thm}[pro]{Theorem}
\newtheorem{de}[pro]{Definition}
\newtheorem{co}[pro]{Comment}
\newtheorem{no}[pro]{Notation}
\newtheorem{vb}[pro]{Example}
\newtheorem{vbn}[pro]{Examples}
\newtheorem{gev}[pro]{Corollary}

\begin{abstract}
We prove a uniformly continuous linear extension principle in topological vector spaces from which we derive a very short and canonical construction of the Lebesgue integral of Banach space valued maps on a finite measure space. The Vitali Convergence Theorem and the Riesz-Fischer Theorem are shown to follow as easy consequences from our construction.
\end{abstract}

\section{Introduction and motivation}

Since the birth of Lebesgue's theory of integration, various introductions and modifications of the theory have been proposed in an attempt to
 \begin{enumerate}
  \item make it more elementary (\cite{Mikusinski}),
  \item unify it with Riemann's approach to integration (\cite{Gordon},\cite{Kurzweil},\cite{McShane},\cite{Schurle}),
  \item extend it to a Banach space valued context (\cite{Birkhoff},\cite{Cascales},\cite{Rodriguez}),
  \item present it in the abstract setting of functional analysis (\cite{BourbakiInt},\cite{Lang}).
 \end{enumerate}
 
In this paper we shall present a very short and canonical construction of the Lebesgue integral via a uniformly continuous linear extension principle in topological vector spaces (Theorem \ref{thm:LETVS}) which addresses the above mentioned topics. The Vitali Convergence Theorem and the Riesz-Fischer Theorem are shown to follow as easy consequences from our construction.

\section{Uniformly continuous linear extension in TVS}

The following lemma is a well known result in the theory of uniform spaces, see e.g. \cite{BourbakiTop}.

\begin{lem}\label{lem:ExtensionUniformSpaces}
Let $X$ be a uniform space, $A \subset X$ a dense subset and $f$ a uniformly continuous map of $A$ into a complete Hausdorff uniform space. Then there exists a unique uniformly continuous extension of $f$ to $X$.
\end{lem}

We say that a collection $\mathcal{K}$ of subsets of a complex vector space is {\em closed under the formation of finite linear combinations} iff the set 
\begin{displaymath}
\alpha K_1 + \beta K_2 = \left\{\alpha x + \beta y \mid x \in K_1, y \in K_2\right\} 
\end{displaymath}
belongs to $\mathcal{K}$ for all $K_1,K_2 \in \mathcal{K}$ and $\alpha,\beta \in \mathbb{C}$. We now apply Lemma \ref{lem:ExtensionUniformSpaces} to obtain the following uniformly continuous linear extension principle in complex topological vector spaces (TVS).
 
\begin{thm}\label{thm:LETVS}{\em(Uniformly continuous linear extension principle in TVS)}
Let $E$ be a TVS, $F \subset E$ a vector subspace, $\mathcal{K}$ a collection of sets $K \subset F$ which covers $F$ and is closed under the formation of finite linear combinations. Let $\lambda$ be a linear map of $F$ into a complete Hausdorff TVS $E^\prime$ which is uniformly continuous on each $K \in \mathcal{K}$. Then $\widetilde{F} = \cup_{K \in \mathcal{K}} \overline{K}$ is a vector space containing $F$ and there exists a unique linear extension of $\lambda$ to $\widetilde{F}$ which is uniformly continuous on the closure of each $K \in \mathcal{K}$.
\end{thm}

\begin{proof}
For $x,y \in \widetilde{F}$ and $\alpha,\beta \in \mathbb{C}$, choose $K_1, K_2 \in \mathcal{K}$ such that $x \in \overline{K_1}$ and $y \in \overline{K_2}$. Then 
\begin{displaymath}
\alpha x + \beta y \in \alpha \overline{K_1} + \beta \overline{K_2} \subset \overline{\alpha K_1 + \beta K_2} \subset \widetilde{F}
\end{displaymath}
the latter inclusion being a consequence of the fact that $\mathcal{K}$ is closed under the formation of finite linear combinations. We conclude that $\widetilde{F}$ is a vector space, which contains $F$ because $\mathcal{K}$ covers $F$. Furthermore, for each $K \in \mathcal{K}$, Lemma \ref{lem:ExtensionUniformSpaces} shows that  $\lambda \mid_K$ extends uniquely to a uniformly continuous map $\overline{\lambda \mid_K}$ of $\overline{K}$ into $E^\prime$. For $x \in \widetilde{F}$, choose $K \in \mathcal{K}$ such that $x \in \overline{K}$ and put $\widetilde{\lambda}(x) = \overline{\lambda \mid_K}(x)$. The assignment $\widetilde{\lambda}$ is well-defined. Indeed, suppose that $x \in \overline{K_1} \cap \overline{K_2}$ with $K_1,K_2 \in \mathcal{K}$ and let $\mathcal{V}$ be the neighbourhood filter of $0$ in $E$. Then, for each $V \in \mathcal{V}$, choose $y_V \in K_1 \cap \left(x + V\right)$ and $z_V \in K_2 \cap \left(x + V\right)$. Notice that $(y_V)_{V \in \mathcal{V}}$ (resp. $\left(z_V\right)_{V \in \mathcal{V}}$) is a net in $K_1$ (resp. $K_2$) converging to $x$. Now
\begin{align*}
& \overline{\lambda \mid_{K_1}}(x) - \overline{\lambda \mid_{K_2}}(x)\\
& \quad (\overline{\lambda \mid_{K_1}} \textrm{ and } \overline{\lambda \mid_{K_2}} \textrm{ are continuous})\\
& = \lim \overline{\lambda \mid_{K_1}}(y_V) - \lim \overline{\lambda \mid_{K_2}}(z_V)\\
& = \lim \lambda(y_V) - \lim \lambda(z_V)\\
& = \lim \left[\lambda(y_V) - \lambda(z_V)\right]\\
& = \lim \lambda(y_V - z_V)\\
& \quad (K_1 - K_2 \in \mathcal{K})\\
& = \lim \overline{\lambda \mid_{K_1 - K_2}}(y_V - z_V)\\
& \quad (\overline{\lambda \mid_{K_1 - K_2}} \textrm{ is continuous})\\ 
& = 0
\end{align*}
establishing the fact that $\widetilde{\lambda}$ is well-defined. The linearity of $\widetilde{\lambda}$ is proved analogously. We conclude that $\widetilde{\lambda}$ has the desired properties and we are done.
\end{proof}

\section{Construction of the Lebesgue integral}

In this section we give a brief outline of how a very short and canonical construction of the Lebesgue integral can be obtained via Theorem \ref{thm:LETVS}.

Let $\Omega = (\Omega,\mathcal{A},\mu)$ be a finite measure space, $E = (E,\|\cdot\|)$ a complex separable Banach space and ${\bf M}(\Omega,E)$ the TVS of Borel measurable maps $f$ of $\Omega$ into $E$, equipped with the topology of convergence in measure. That is, the sets 
\begin{displaymath}
{\bf V}_{\epsilon} = \left\{ f \in {\bf M}\left(\Omega,E\right) \mid \mu\left(\left\{\|f\| \geq \epsilon\right\}\right) < \epsilon\right\}, \phantom{1} \epsilon > 0, 
\end{displaymath}
constitute a base for the neighbourhood filter of $0$. It is well known that the uniform structure of ${\bf M}\left(\Omega,E\right)$ is complete and pseudometrisable, see e.g. \cite{BourbakiInt}. Unless otherwise stated, all subsets of ${\bf M}(\Omega,E)$ are equipped with the uniformity of convergence in measure.

A map $s \in {\bf M}(\Omega,E)$ is called {\em simple} iff there exists a finite measurable partition $A_1,\ldots,A_n$ of $\Omega$ such that $s$ assumes a unique value $s_i \in E$ on each $A_i$. We denote the collection of simple maps as ${\bf S}(\Omega,E)$. Notice that ${\bf S}(\Omega,E)$ is a vector subspace of ${\bf M}(\Omega,E)$. We define the {\em integral of $s \in {\bf S}(\Omega,E)$} as 
\begin{displaymath}
\int s = \sum_{i} \mu(A_i) s_i
\end{displaymath}
with $A_1,\ldots,A_n$ the measurable partition and $s_1,\ldots,s_n$ the values associated with $s$. 

\begin{pro}\label{pro:PropertiesIntegralOfSimpleMaps}
The mapping $\int$ of  ${\bf S}(\Omega,E)$ into $E$ is linear and, if $s \in {\bf S}(\Omega,E)$, then $\|s\| \in {\bf S}(\Omega,\mathbb{C})$ and $\|\int s\| \leq \int \|s\|$. In particular, if $E = \mathbb{C}$, then the mapping $\int$ of  ${\bf S}(\Omega,\mathbb{C})$ into $\mathbb{C}$ is positive in the sense that $\int s \geq 0$ if $s \geq 0$ and monotonic in the sense that $\int s \geq \int t$ if $s \geq t$.
\end{pro}

\begin{proof}
This is standard.
\end{proof}

We call a set ${\bf E} \subset {\bf S}(\Omega,E)$ {\em elementary} iff for each $\epsilon > 0$ there exists $\delta > 0$ such that $\int \|s\| 1_A < \epsilon$ whenever $s \in {\bf E}$ and $A \in \mathcal{A}$ with $\mu(A) < \delta$. The collection of elementary sets in ${\bf S}(\Omega,E)$ is denoted as $\mathcal{E}(\Omega,E)$.

\begin{pro}\label{pro:PropertiesElementarySets}
$\mathcal{E}(\Omega,E)$ contains all sets ${\bf E} \subset {\bf S}(\Omega,E)$ which are totally bounded under the weak uniformity for the mapping $\int \circ \|\cdot\|$ of ${\bf S}(\Omega,E)$ into $\mathbb{C}$. Furthermore, $\mathcal{E}(\Omega,E)$ is closed under the formation of finite linear combinations.
\end{pro}

\begin{proof}
Let ${\bf E} \subset {\bf S}(\Omega,E)$ be totally bounded under the weak uniformity for the mapping $\int \circ \|\cdot\|$ of ${\bf S}(\Omega,E)$ into $\mathbb{C}$. In order to show that ${\bf E} \in \mathcal{E}(\Omega,E)$, fix $\epsilon > 0$. Then there exists a finite set ${\bf E}_0 \subset {\bf S}(\Omega,E)$ such that for all $s \in {\bf E}$
\begin{eqnarray}
\min_{t \in {\bf E}_0}\int \|s - t\| \leq \epsilon.\label{eq1}
\end{eqnarray}
Now, by (\ref{eq1}), for $s \in {\bf E}$ and $A \in \mathcal{A}$,
\begin{displaymath}
\min_{t \in {\bf E}_0} \left|\int \|s\|1_A - \int \|t\| 1_A\right| \leq \min_{t \in {\bf E}_0} \int \left|\|s\| - \|t\|\right|1_A \leq \min_{t \in {\bf E}_0} \int \|s - t\| \leq \epsilon
\end{displaymath}
whence
\begin{eqnarray}
\int \|s\|1_A \leq \max_{t \in {\bf E}_0}\int \|t\|1_A + \epsilon.\label{eq2}
\end{eqnarray}
Choose a constant $C > 0$ such that
\begin{eqnarray*}
\mymax_{t \in {\bf F}} \sup_{x \in \Omega} \left|t(x)\right| \leq C.
\end{eqnarray*}
Then, by (\ref{eq2}), for $s \in {\bf E}$ and $A \in \mathcal{A}$,
\begin{displaymath}
\int \|s\|1_A \leq \int \|s\|1_A \leq \max_{t \in {\bf F}}\int \|t\|1_A + \epsilon \leq C \mu(A) + \epsilon
\end{displaymath}
entailing that ${\bf E} \in \mathcal{E}(\Omega,E)$. We now show that $\mathcal{E}(\Omega,E)$ is closed under the formation of finite linear combinations. Fix ${\bf E}_1,{\bf E}_2 \in \mathcal{E}(\Omega,E)$ and $\alpha,\beta \in \mathbb{C}$. Then, for $s_1 \in {\bf E}_1$, $s_2 \in {\bf E}_2$ and $A \in \mathcal{A}$,
\begin{displaymath}
\int \|\alpha s_1 + \beta s_2\|1_A \leq \int \left(\left|\alpha\right|\|s_1\| + \left|\beta\right|\|s_2\|\right)1_A = \left|\alpha\right| \int \|s_1\|1_A + \left|\beta\right|\int \|s_2\|1_A
\end{displaymath}
immediately yielding that $\alpha {\bf E}_1 + \beta {\bf E}_2 \in \mathcal{E}\left(\Omega,E\right)$.
\end{proof}

\begin{pro}\label{pro:InitialUnif}
The uniformity of convergence in measure on ${\bf S}(\Omega,E)$ is weaker than the weak uniformity for the mapping $\int \circ \|\cdot\|$ of ${\bf S}(\Omega,E)$ into $\mathbb{C}$, and these uniformities coincide on elementary sets. In particular, the mapping $\int$ of ${\bf S}(\Omega,E)$ into $E$ is uniformly continuous on elementary sets.
\end{pro}

\begin{proof}
For $s , t \in {\bf S}(\Omega,E)$ and $\epsilon > 0$, the inequalities
\begin{align}
& \mu(\left\{\|s -t\| \geq \epsilon\right\}) \leq \epsilon^{-1} \int \|s -t\|,\label{eq3}\\
&\int \|s - t\| \leq \epsilon \mu(\Omega) + \int \|s - t\| 1_{\left\{\|s - t\| \geq \epsilon\right\}},\label{eq4}\\
&\left\|\int s - \int t\right\| \leq \int \|s - t\|\label{eq5}
\end{align}
are easily established. Now (\ref{eq3}) implies that the uniformity of convergence in measure on ${\bf S}(\Omega,E)$ is weaker than the weak uniformity for the mapping $\int \circ \|\cdot\|$ of ${\bf S}(\Omega,E)$ into $\mathbb{C}$, and (\ref{eq4}) entails that these uniformities coincide on elementary sets. Finally, (\ref{eq5}) shows that the mapping $\int$ of ${\bf S}(\Omega,E)$ into $E$ is uniformly continuous on elementary sets.
\end{proof}

A map $f \in {\bf M}(\Omega,E)$ is called {\em (Lebesgue) integrable} iff it belongs to ${\bf L}(\Omega,E) = \cup_{ {\bf E} \in \mathcal{E}(\Omega,E)} \overline{{\bf E}}$. From Theorem \ref{thm:LETVS} and the previous propositions we conclude that ${\bf L}(\Omega,E)$ is a vector space containing ${\bf S}(\Omega,E)$ and that there exists a unique linear extension of $\int$ to ${\bf L}(\Omega,E)$ which is uniformly continuous on the closure of each elementary set. We denote this extension again as $\int$ and we define the {\em integral of $f \in {\bf L}(\Omega,E)$} as $\int f$.

Let $\stackrel{\mu}{\rightarrow}$ stand for convergence in measure.

\begin{pro}\label{pro:PropertiesIntegralOfIntegrableMaps}
The mapping $\int$ of  ${\bf L}(\Omega,E)$ into $E$ is linear and, if $f \in {\bf L}(\Omega,E)$, then $\|f\| \in {\bf L}(\Omega,\mathbb{C})$ and $\|\int f\| \leq \int \|f\|$. In particular, if $E = \mathbb{C}$, then the mapping $\int$ of  ${\bf L}(\Omega,\mathbb{C})$ into $\mathbb{C}$ is positive in the sense that $\int f \geq 0$ if $f \geq 0$ and monotonic in the sense that $\int f \geq \int g$ if $f \geq g$. That is, Proposition \ref{pro:PropertiesIntegralOfSimpleMaps} continues to hold if we replace ${\bf S}(\Omega,E)$ by ${\bf L}(\Omega,E)$ and ${\bf S}(\Omega,\mathbb{C})$ by ${\bf L}(\Omega,\mathbb{C})$.
\end{pro}

\begin{proof}
The linearity of $\int$ was already established. Fix $f \in {\bf L}(\Omega,E)$. We prove that $\|f\| \in {\bf L}(\Omega,\mathbb{C})$ and $\|\int f \| \leq \int \|f\|$. By definition of ${\bf L}(\Omega,E)$, there exists a sequence $\left(s_n\right)_n$ such that $\{s_n \mid n\} \in \mathcal{E}(\Omega,E)$ and $s_n \stackrel{\mu}{\rightarrow} f$. One easily sees that $\{\|s_n\| \mid n\} \in \mathcal{E}(\Omega,\mathbb{C})$ and that $\|s_n\| \stackrel{\mu}{\rightarrow} \|f\|$. Hence $\|f\| \in {\bf L}(\Omega,\mathbb{C})$ and, $\int$ being continuous on the closure of each elementary set, Proposition \ref{pro:PropertiesIntegralOfSimpleMaps} gives 
\begin{align*}
\left\|\int f\right\| = \lim_n \left\|\int s_n\right\| \leq \lim_n \int \|s_n\| = \int \|f\|. 
\end{align*}
The other properties are now easily established.
\end{proof}

We call a set ${\bf F} \subset {\bf L}(\Omega,E)$ {\em uniformly integrable} iff for each $\epsilon > 0$ there exists $\delta > 0$ such that $\int \|f\| 1_A < \epsilon$ whenever $f \in {\bf F}$ and $A \in \mathcal{A}$ with $\mu(A) < \delta$. Notice that the closure of an elementary set is uniformly integrable because $\int$ is continuous on such a set. The collection of uniformly integrable sets in ${\bf L}(\Omega,E)$ is denoted as $\mathcal{I}(\Omega,E)$.

\begin{pro}\label{pro:PropertiesUnifIntSets}
$\mathcal{I}(\Omega,E)$ contains all sets ${\bf F} \subset {\bf S}(\Omega,E)$ which are totally bounded under the weak uniformity for the mapping $\int \circ \|\cdot\|$ of ${\bf L}(\Omega,E)$ into $\mathbb{C}$. Furthermore, $\mathcal{I}(\Omega,E)$ is closed under the formation of finite linear combinations. That is, Proposition \ref{pro:PropertiesElementarySets} continues to hold if we replace $\mathcal{E}(\Omega,E)$ by $\mathcal{I}(\Omega,E)$ and ${\bf S}(\Omega,E)$ by ${\bf L}(\Omega,E)$.
\end{pro}

\begin{proof}
Let ${\bf F} \subset {\bf L}(\Omega,E)$ be totally bounded under the weak uniformity for the mapping $\int \circ \|\cdot\|$ of ${\bf L}(\Omega,E)$ into $\mathbb{C}$. In order to show that ${\bf F} \in \mathcal{I}(\Omega,E)$, fix $\epsilon > 0$. Copying the first part of the proof of Proposition \ref{pro:PropertiesElementarySets}, we find a finite set ${\bf F}_0 \subset {\bf L}(\Omega,E)$ such that for each $f \in {\bf F}$ and each $A \in \mathcal{A}$ 
\begin{eqnarray}
\int \|f\|1_A \leq \max_{g \in {\bf F}_0} \int \|g\|1_A + \epsilon.\label{eq:UnifIntF}
\end{eqnarray}
By definition of ${\bf L}(\Omega,E)$, there exists, for each $g \in {\bf F}_0$, a set ${\bf E}_g \in \mathcal{E}(\Omega,E)$ such that $g \in 
\overline{{\bf E}_g}$. Now (\ref{eq:UnifIntF}), in combination with the fact that $\cup_{g \in {\bf F}_0} {\bf E}_g \in \mathcal{E}(\Omega,E)$ and the continuity of $\int$ on the closure of each elementary set, easily implies that ${\bf F} \in \mathcal{I}(\Omega,E)$. As in the proof of Proposition \ref{pro:PropertiesElementarySets}, the fact that $\mathcal{I}(\Omega,E)$ is closed under the formation of finite linear combinations is a straightforward application of the linearity and the monotonicity of $\int$.
\end{proof}

\begin{pro}\label{pro:InitialUnifL}
The uniformity of convergence in measure on ${\bf L}(\Omega,E)$ is weaker than the weak uniformity for the mapping $\int \circ \|\cdot\|$ of ${\bf L}(\Omega,E)$ into $\mathbb{C}$, and these uniformities coincide on uniformly integrable sets. In particular, the mapping $\int$ of ${\bf L}(\Omega,E)$ into $E$ is uniformly continuous on uniformly integrable sets. That is, Proposition \ref{pro:InitialUnif} continues to hold if we replace ${\bf S}(\Omega,E)$ by ${\bf L}(\Omega,E)$ and  `elementary set' by `uniformly integrable set'. 
\end{pro}

\begin{proof}
This is identical to the proof of Proposition \ref{pro:InitialUnif}.
\end{proof}

\begin{gev}\label{cor:SD}
The space ${\bf L}(\Omega,E)$, equipped with the weak uniformity for the mapping $\int \circ \|\cdot\|$ of ${\bf L}(\Omega,E)$ into $\mathbb{C}$, contains ${\bf S}(\Omega,E)$ as a dense subspace.
\end{gev}

\begin{proof}
Fix $f \in {\bf L}(\Omega,E)$. Then, by definition of ${\bf L}(\Omega,E)$, there exists a sequence $(s_n)_n$ such that $\{s_n \mid n\} \in \mathcal{E}(\Omega,E)$ and $s_n \stackrel{\mu}{\rightarrow} f$. We know that $\{s_n \mid n\} \cup \{f\} \in \mathcal{I}(\Omega,E)$ and thus Proposition \ref{pro:InitialUnifL} implies that, on the set $\{s_n \mid n\} \cup \{f\}$, the uniformity of convergence in measure coincides with the weak uniformity for the mapping $\int \circ \|\cdot\|$ of ${\bf L}(\Omega,E)$ into $E$. Hence $\int \|f - s_n\| \rightarrow 0$, finishing the proof of the corollary. 
\end{proof}

\begin{gev}\label{cor:Vit} (Vitali) 
Fix $f \in {\bf M}(\Omega,E)$ and $\left(f_n\right)_n$ in ${\bf L}(\Omega,E)$. Then the following are equivalent.
\begin{enumerate}
	\item $f \in {\bf L}(\Omega,E)$ and $\int \|f - f_n\| \rightarrow 0$.
	\item $\left\{f_n \mid n\right\} \in \mathcal{I}(\Omega,E)$ and $f_n \stackrel{\mu}{\rightarrow} f$. 
\end{enumerate}
\end{gev}

\begin{proof}
$(1) \Rightarrow (2)$ Suppose that $f \in {\bf L}(\Omega,E)$ and $\int \|f - f_n\| \rightarrow 0$. Then $\{f_n \mid n\}$ is totally bounded under the weak uniformity for the mapping $\int \circ \|\cdot\|$ of ${\bf L}(\Omega,E)$ into $\mathbb{C}$ and thus, by Proposition \ref{pro:PropertiesUnifIntSets}, $\{f_n \mid n\} \in \mathcal{I}(\Omega,E)$. Furthermore, by Proposition \ref{pro:InitialUnifL}, the uniformity of convergence in measure on ${\bf L}(\Omega,E)$ is weaker than the weak uniformity for the mapping $\int \circ \|\cdot\|$ of ${\bf L}(\Omega,E)$ into $\mathbb{C}$. Hence we conclude that $f_n \stackrel{\mu}{\rightarrow} f$.
$(2) \Rightarrow (1)$ Suppose that $\left\{f_n \mid n\right\} \in \mathcal{I}(\Omega,E)$ and $f_n \stackrel{\mu}{\rightarrow} f$. Corollary \ref{cor:SD} allows us to choose, for $n \in \mathbb{N}_0$, $s_n \in {\bf S}(\Omega,E)$ such that $\int \|f_n - s_n\| \leq 1/n$. Now one easily establishes that $\left\{s_n \mid n\right\} \in \mathcal{E}(\Omega,E)$, whence  $f \in \overline{\left\{s_n \mid n\right\}} \subset {\bf L}(\Omega,E)$. Furthermore, Proposition \ref{pro:InitialUnifL} reveals that, on the set $\{f_n \mid n\} \cup \{f\}$, the uniformity of convergence in measure coincides with the weak uniformity for the mapping $\int \circ \|\cdot\|$ of ${\bf L}(\Omega,E)$ into $\mathbb{C}$. As a consequence, $\int \|f - f_n\| \rightarrow 0$.
\end{proof}

\begin{gev}(Riesz-Fischer)
The space ${\bf L}(\Omega,E)$, equipped with the weak uniformity for the mapping $\int \circ \|\cdot\|$ of ${\bf L}(\Omega,E)$ into $\mathbb{C}$, is complete.
\end{gev}

\begin{proof}
Let $\left(f_n\right)_n$ be Cauchy in ${\bf L}(\Omega,E)$. Proposition \ref{pro:InitialUnifL} entails that the uniformity of convergence in measure is weaker than the weak uniformity for the mapping $\int \circ \|\cdot\|$ of ${\bf L}(\Omega,E)$ into $\mathbb{C}$. In particular, $(f_n)_n$ is also Cauchy for the uniformity of convergence in measure. The completeness of ${\bf M}(\Omega,E)$ allows us to find $f \in {\bf M}(\Omega,E)$ such that $f_n \stackrel{\mu}{\rightarrow} f$. Now, $\left(f_n\right)_n$ being Cauchy in ${\bf L}(\Omega,E)$, the set $\{f_n \mid n\}$ is totally bounded under the weak uniformity for the mapping $\int \circ \|\cdot\|$ of ${\bf L}(\Omega,E)$ into $\mathbb{C}$ and thus, by Proposition \ref{pro:PropertiesUnifIntSets}, $\{f_n \mid n\} \in \mathcal{I}(\Omega,E)$. Finally, an application of Corollary \ref{cor:Vit} reveals that $f \in {\bf L}(\Omega,E)$ and $\int \|f - f_n\| \rightarrow 0$. 
\end{proof}

\end{document}